\documentclass[12pt,twoside,letter]{amsart}
\usepackage{txfonts}
\usepackage{comment}
\usepackage{bbm}
\usepackage{mathrsfs}
\usepackage{amsfonts}
\usepackage{amsmath}
\usepackage{amssymb}
\usepackage{wasysym}
\usepackage{psfrag}
\usepackage{graphics,epsfig,amsmath}  
\usepackage{comment}
\usepackage{enumerate}
\usepackage{color}
\usepackage{eufrak}

\newtheorem{theorem}{Theorem}
\newtheorem{definition}[theorem]{Definition}

\newtheorem{lemma}[theorem]{Lemma}
\newtheorem{proposition}[theorem]{Proposition}
\newtheorem{preremark}[theorem]{Remark}
\newenvironment{remark}{\begin{preremark}\rm}{\end{preremark}}

\def\M{\mathcal{M}}

\def\real{{ \mathbb R}}
\def\dist{ {\text dist}}

\title[On an example of F. Almgren and H. Federer]
{On a remarkable example of F. Almgren and H. Federer in 
the global theory of minimizing geodesics}

\author[X. Su] {Xifeng Su}

\thanks{X. S supported by National Natural Science Foundation of China (Grant No. 11301513) and ``the Fundamental Research Funds for the Central Universities"}

\address{School of Mathematical Sciences\\
Beijing Normal University\\
No. 19, XinJieKouWai St.,HaiDian District\\
 Beijing 100875, P. R. China}
\email{xfsu@bnu.edu.cn}

\author[R. de la Llave]{Rafael de la Llave}

\thanks{R.L. has been partially supported by
NSF grant DMS-1800241. Progress was made while 
R.L. was visiting the JLU-GT institute for theoretical Science
and Beijing Normal University. 
The final version was written 
while R.L.
was in residence at the Mathematical Sciences Research Institute in Berkeley, California, during the Fall 2018 semester supported by DMS-1440140 }

\address{School of Mathematics\\
Georgia Institute of Technology \\
686 Cherry St. \\
Atlanta GA 30332, USA}
\email{rafael.delallave@math.gatech.edu}

\begin{document}

\today

\maketitle

\begin{abstract}
We present an exposition of 
a  remarkable example attributed to Frederick Almgren Jr. in \cite[Section 5.11]{Federer74} to illustrate the need of certain definitions 
in the calculus of variations. 

The Almgren-Federer example,  besides its
intended goal of illustrating subtle aspects of 
geometric measure theory,
is also a problem in the theory of geodesics.  Hence, we wrote 
an exposition of the beautiful ideas of Almgren and Federer from 
the point of view of geodesics.

In the language of geodesics, Almgren-Federer example constructs  metrics in
$\mathbb{S}^1\times \mathbb{S}^2$,
with the property  that none of the  Tonelli geodesics 
(geodesics which minimize the length in a homotopy class) 
are Class-A  minimizers in the sense of Morse
(any finite length segment in the universal cover 
minimizes the length between the end points; this is also 
sometimes given other names). In other words, even if a curve is a
minimizer of length among all the curves homotopic to it, 
by repeating it enough times, we get a closed curve which does 
not minimize in its homotopy class.

In that respect, the example is more dramatic than a better known example 
due to Hedlund of 
a metric in  $\mathbb{T}^3$ for which only 3 Tonelli minimizers (and 
their multiples) are Class-A minimizers. 

For dynamics, the example also illustrates different definitions of ``integrable'' and clarifies the relation between minimization and hyperbolicity and 
its interaction with topology. 
\end{abstract}




\section{introduction}

The paper \cite{Federer74}, lays the foundation of the theory of flat
chains. In Section 5.11 it presents a very remarkable example
(attributed there to F. Almgren Jr. and which we will henceforth  call 
Almgren-Federer example).

 The role of the Almgren-Federer example in
\cite{Federer74} is to illustrate the need of considering rather
general objects (flat chains) in the calculus of variations in geometric measure
theory. The theory developed in \cite{Federer74} shows
that certain functionals have minimizers in 
the set of flat chains. The example shows 
that the minimizers cannot be much simpler objects.

This Almgren-Federer example is a geometric problem and its
properties are illuminating also for several other versions of
calculus of variations (e.g. for the theory of geodesics and for
problems in Hamiltonian dynamics).  It can also serve as motivation 
for some of the definitions in theory of minimizing measures
\cite{Mather91, Mane90}. Tentatively, we also believe it 
sheds some light in the problem of homogeneization on periodic media
or the problems of statistical mechanics in quasi-crystals.

The extremely precise and beautiful (but multilayered) language of geometric measure theory may be an
impediment for many readers to appreciate the beauty of the example
and to appreciate its role on other contexts in which it is 
also relevant.  The goal of this paper is to present the ideas 
behind the example in a way that it is accessible to practioners in 
other fields (e.g. mechanics or geodesic flows)  for 
which the example is highly relevant.  In particular, we present 
applications to the theory of geodesics.

In the language of the theory of geodesics, the example constructs metrics 
in 
 $\mathbb{S}^1\times \mathbb{S}^2$, for  which 
none of the  Tonelli geodesics 
(periodic geodesics which minimize the length  in a homotopy class, see 
Appendix~\ref{sec:tonelli}) is a 
Class-A  minimizer in the sense of \cite{Morse24}
 (any finite length segment of the lift of the orbit to the universal cover 
minimizes the length between the two end points of the segment.)~\footnote{In the literature, this property is referred by many other 
names such as minimizer, local minimizer, global minimizer, etc. 
Unfortunately, these names are not used consistently for 
the same concepts by different authors, so we 
prefer to revert to the names used in \cite{Morse24}. }
See Appendix~\ref{sec:morse}.

It is interesting to compare the Almgren-Federer example with the 
better known Hedlund example \cite{Hedlund32,Levi97}. In the Hedlund example, 
three Tonelli geodesics (and their multiples) 
 are Class-A, whereas in the Almgren-Federer example, \emph{none} 
of the Tonelli minimizers is Class-A.
As a matter of fact, in the Almgren-Federer example, 
we give a characterization  of the Class-A geodesics, none of which is periodic.
Of course, the mechanism in Almgren-Federer example is 
very different from the mechanism in \cite{Hedlund32}.
\smallskip

The paper \cite{Levi97} studies the global dynamical properties of
Hedlund example. This study shows that there are geodesics with surprising properties (e.g. remaining close to each of the minimizers for very long segments).
The paper \cite{Bangert88} presents a reworking of the results of
\cite{Morse24, Hedlund32} on 2-D manifolds  from a new 
very  powerful point of view
that also allows to obtain many results on twist mappings and solid state models
that were originally obtained by \cite{Mather82,AubryD}.

\subsection{Organization of the paper}

In Section~\ref{sec:description} we present the Almgren-Federer 
example and call attention some  geometric properties. 
In Section~\ref{sec:statements} we state formally the main results of 
the paper: the fact that Tonelli orbits are not Class-A and 
the characterization of Class-A geodesics for the example.
We also provide some heuristic intuition on why the results should be true. 

The proofs of the results are presented in 
Sections~\ref{sec:tonelliproof} and~\ref{sec:classA}.
 Given that the aim of 
the paper is pedagogical, besides the precise proofs we have included many heuristic comments. 
We hope that they are not too distracting, but if they are, 
the reader is encouraged to skip them. 

In the final Section~\ref{sec:speculation} we indulge in some heuristic 
explanation of the possible relevance of the ideas in Almgren-Federer 
example in other areas such 
as quasi-periodic media.

In Appendix~\ref{sec:geodesic}, we collect some of
the  theory of geodesic flows. It turns out that the
Almgren-Federer example also illustrates the relation between
different notions of integrability in Hamiltonian systems.

In Appendices~\ref{sec:morse} and ~\ref{sec:tonelli}, we present the classical 
definitions of Tonelli and Morse and some of the classical results on 
global calculus of variations of paths. This material is, of course 
classic, but we have included notes. 
In this paper, we will use mainly methods of calculus of variations, 
but in some points, we will make reference to the geodesic flow. 
Standard references  are  \cite{Paternain99,AbrahamM78,ArnoldA68}.
Some analysis at the Almgren-Federer example based on properties of the geodesic flow is in Appendix~\ref{sec:mather}.


\section{Geometric description of Almgren-Federer  example}

\label{sec:description}

Following the presentation in \cite{Federer74} we consider the sphere $\mathbb{S}^2$, written explicitly as: 
\begin{equation} \label{coordinates} 
\left\{~ (y, z) ~\big| ~ y\in [-1,1], z\in \mathbb{C}, y^2 + |z|^2 = 1 \right\}.
\end{equation} 

We also consider the circle $\mathbb{S}^1$ which we give 
the coordinate $x$.

We consider the manifold 
$\widetilde{\M}
 = \mathbb{R}\times \mathbb{S}^2$, as well as 
the quotient manifold 
\[
\M = \mathbb{R} \times \mathbb{S}^2 / \sim
\]where 
\begin{equation} 
\label{equivalence} 
(x+1,  y, z) \sim (x, y, z\ e^{2\pi i \omega})
\end{equation} 
for some $\omega\in \mathbb{R} - \mathbb{Q}$.

Clearly, the manifold $\widetilde{ \M}$ is the universal cover 
of $\M$ and $\M$ is diffeomorphic to 
$\mathbb{S}^1 \times \mathbb{S}^2$.  The only homology (or homotopy)
of $\M$ is associated to the $\mathbb{S}^1$ factor.

\begin{remark} 
For dynamicists, the manifold $\M$ is the suspension manifold
\cite{KatokH95, Robinson12}
  of 
the rotation in $\mathbb{S}^2$ given in the coordinates in 
\eqref{coordinates} by 
\begin{equation}  \label{rotation} 
(y,z) \mapsto (y, z\ e^{2 \pi i \omega}). 
\end{equation}

Following this intuition, for people with background in dynamics, it 
may be useful to imagine $\M$ as 
$\mathbb{S}^2$ evolving in time -- time  is the variable $x$ in
the $\mathbb{S}^1$ factor -- but coming back to itself rotated 
by the irrational rotation \eqref{rotation}.  Of course, other 
physical interpretations may be useful.

The effects of the example, may be understood in this interpretation
because we are
trying to approximate  a problem involving an irrational rotation
by periodic orbits. 
\end{remark} 

\begin{remark} 
We can consider $x, y, z$ coordinates of the manifolds also as functions (\cite{Federer74} makes a distinction between the coordinates and the 
functions returning the value of the coordinates,  but we will not 
make this distinction).
\end{remark} 

\begin{remark} The paper \cite{Federer74} formulates 
the quotient \eqref{equivalence}  slightly differently.
It fixes  $\omega  = 1/(2 \pi)$. 
\end{remark}

To describe the geometry in pictorial terms it is convenient to 
think of $\mathbb{S}^2$ embedded in $\mathbb{R}^3$ in such a way 
that $y$ is the vertical coordinate. 
We can use the geographical  notation and 
we will call the set $\{y= 0\}$, the equator and the sets
 $\{ y = cte\}$, the parallels. 
We will also find it convenient 
to introduce the longitude $\sigma$ by 
$z = (1 - y^2)^{1/2}e^{i \sigma}$. Note that in a parallel of 
constant $y$, the longitude ranges over the interval $[0, 2 \pi)$, 
with both ends identified. 

We will think of the coordinate $x$ as a time that advances.

\bigskip 

The paper \cite{Federer74} endows $\mathbb{R}\times \mathbb{S}^2$, 
  the universal cover of 
$\M$,  with a metric 
\begin{equation}\label{AFmetric} 
(1 +y^2)( dx^2 + dy^2 + dz^2)^{1/2}.
\end{equation}

The metric  \eqref{AFmetric} is invariant 
under rotations around the $y$-axis.
 It is also symmetric under translations in the $x$
axis. Hence, it is invariant  under the equivalence~\eqref{equivalence}
and, hence, it defines  a metric on $\M$.

Note that the coordinates $x, y, z$
are an orthogonal system of coordinates.

The important property for us  is that for a fixed $x$, the resulting
$\mathbb{S}^2$ has an  \emph{``hourglass''} shape. 
The length of a parallel at height $y \in [-1,1]$, is 
\[
\begin{split} 
2 \pi (1+y^2)\sqrt{(1 - y^2)} &= 2 \pi\sqrt{ (1+y^2)(1+y^2)(1 - y^2) } \\
&= 2 \pi \sqrt{ 1 + y^2 -y^4 -y^6}.
\end{split}
\]
The length of the parallel  starts growing with $y$ 
from the equator and then decreases. 
The same property will happen if we consider the standard metric in 
$\mathbb{S}^2$ multiplied by a suitable conformal factor depending on $y$.
For the purposes of the analysis here, the  \emph{hourglass shape} is the most important property of
the example. The conclusions are very robust since they depend mostly in
this property and the fact that the geodesics in the equator cannot be periodic.

\begin{figure}[htp]
\centering
\includegraphics[width=7cm,height=7.5cm]{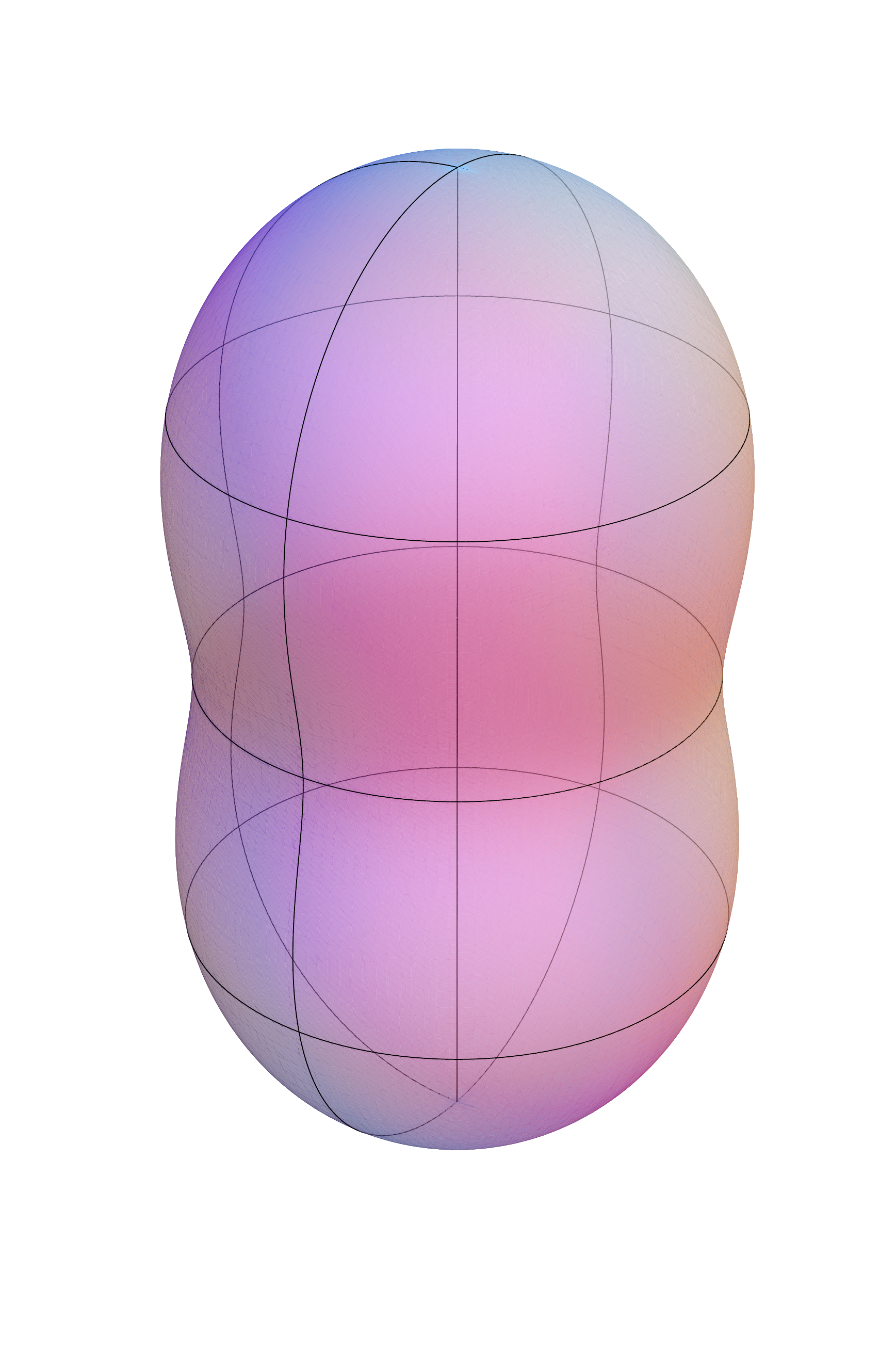}
\caption{Heuristic picture  of  the Almgren-Federer metric for 
a fixed $x$}
\label{eaxamplepicture}
\end{figure}

\section{Statement of results} 
\label{sec:statements} 

The main result on the example in \cite{Federer74} are
the following. We refer to Appendices~\ref{sec:morse} and~\ref{sec:tonelli} 
for the standard definitions  of Tonelli orbits 
(minimizers of length in a homotopy class) and 
the Class-A minimizing orbits (orbits that, in the universal cover 
are minimizers of distance between any pair of points in the orbit). 
See Definitions~\ref{def:tonelli} and~\ref{def:classA}.

\begin{theorem} 
\label{thm:tonelli} 
None of the Tonelli minimizing periodic orbits in 
Example \eqref{AFmetric} are Class-A minimizers.
\end{theorem} 

\begin{remark}
\cite{Hedlund32} constructed an example for 
which there are only three Tonelli Class-A minimizers. 
\end{remark}

It is elementary that 
 a periodic orbit which is Class-A is necessarily a Tonelli minimizer. Nevertheless, the main point of Almgren-Federer example 
is that Tonelli orbits may fail to be Class-A for some metrics. 

The reason for 
Tonelli minimizers failing to be Class-A  
 is that  the definition of 
Tonelli minimizers, requires that the length is minimal 
along the orbits in the same homotopy class. Nevertheless, when we go 
to the universal cover, we could consider multiples of the curve which 
wrap around many times. If some of these multiples is not a Tonelli minimizer
(i.e. one can deform the multiple to get a lower length orbit), then the orbit is not a Class-A minimizer.
Slightly more informally, a periodic orbit is Tonelli when it is stable
(the length cannot be decreased)under perturbations of the same period. To be Class-A it has to be stable   under perturbations of whatever period. 

The idea of the proof of Theorem~\ref{thm:tonelli} is 
to show that if $\gamma$ is a Tonelli minimizer of 
period $T$, it is also a periodic
curve  of period $n T $ for every  $n \in \mathbb{N}$. 
We will show  that there is always an $n$ so that the $nT$ periodic orbit
can be shortened by perturbations of length $nT$.
In other words: the orbits whose length cannot be lowered by perturbations of
a certain period are such that their length can be lowered by perturbations of
longer period.

Actually, we will prove a more general result  than Theorem~\ref{thm:tonelli}. 

\begin{theorem} 
\label{thm:classA} 
The only Class-A minimizing geodesics 
in Example~\eqref{AFmetric} are 
the geodesics  given in the universal cover by: 
\begin{equation}\label{classAorbits}
x(t) = t,\  y(t) = 0,\  z(t) =  z_0 .
\end{equation} 
\end{theorem}

Note that each of the orbits in \eqref{classAorbits} is dense 
on the equator of $\M$. Note that, every time we increment $x$ 
by $1$, the identification \eqref{equivalence} makes the orbit to come
back identified 
with an irrational rotation. Therefore, repeating the
argument, we obtain that the orbits are dense in the equator. 
On the other hand, the speed is constant and concentrated in the direction of 
incrementing only $x$.

Note that the Class-A orbits 
in \eqref{classAorbits} form a one-dimensional family. 
In particular, they are not isolated and hence, cannot be hyperbolic. 
This is in contrast with the Class-A orbits in Hedlund example 
which are hyperbolic and have  very interesting shadowing properties.
These shadowing properties are 
established using methods from hyperbolic systems in 
\cite{Levi97}. (It seems that they could perhaps 
be established by variational methods 
\cite{Bessi, Bolotin,Cheng,Paternain}). 
The relation 
between hyperbolicity and minimization is a very interesting 
question since both properties lead to shadowing, see 
\cite{Offin}.  We also mention that relations between 
minimizing sets and hyperbolicity have been explored for 
generic systems specially in low dimensions \cite{LeCalvez88,
contreras, Arnaud}.

\begin{remark}
It is worth remarking for experts that our considerations 
differ in some aspects from those
of \cite{Federer74} and this affects some of 
the proofs and some of the considerations. 

We consider geodesics, whereas \cite{Federer74} considers flat chains. The main result on existence of minimizers for geodesic is Tonelli's Theorem
which asserts the existence of minimizers in a homotopy class.

Note that in contrast, in other theories of minimization among more
complicated objects (e.g. measures in J. Mather's theory or
flat chains in geometric measure theory), one cannot talk of homotopy, but only about homology. 
See Appendix \ref{sec:mather}.

For manifolds like the ones we consider
in the Almgren-Federer example, homotopy and homology coincide, but
in more general cases, the homotopy and homology can be very different. 

Our proof of Theorem~\ref{Tonelli theorem} is significantly simpler
than the proof of the corresponding results in \cite{Federer74}. The
reason being that the flat chains that minimize a fixed period could,
in principle be more complicated objects than closed geodesics and the
symmetry breaking cannot be and the symmetry breaking requires more
sophisticated constructions. See \cite[Section 5.11]{Federer74}. Of
course, the ideas of the argument presented in this exposition
come from \cite{Federer74}. 

\end{remark}

\section{Proof of  Theorem~\ref{thm:tonelli}}
\label{sec:tonelliproof} 
We will assume that there indeed exists a Tonelli Class-A minimizer and 
then prove that it so has several contradictory properties. Hence, 
no such a thing exists.

If  there were $\gamma_m$ a curve of homotopy 
$m \in \mathbb{Z}$ which is Tonelli and Class-A, 
by repeating it enough, we could get 
an orbit $n\gamma_m$ 
that moves in the universal cover from $x=0,y= y_0 , z_0$ to 
$x = nm, y= y_0, z_0e^{2\pi i n m \omega}$.

We can compare this curve with the curve  going from 
$(0, y_0, z_0)$ to $(0, 0, z_0)$ along a meridian, then 
going from $(0,0, z_0)$ to $(nm, 0, z_0)$ and then going to a $(nm, y_0, z_0)$ 
along a meridian. We see that the length of this curve is less 
than  $nm + 2 B$, where $B$ is an upper bound of 
the length of the segments of meridian used above as the initial 
and final segments.

The length of the $n$ multiple of  $\gamma_m$, 
the Tonelli orbit of homotopy $m$ 
is $n|\gamma_m|$.

 If the Tonelli orbit was a Class-A we 
would have that the length of its  $n$ cover would be smaller than the length 
of the trial orbit described above. Hence, if the Tonelli orbit $\gamma_m$
were Class-A, we would have
\[ 
n|\gamma_m| \le nm + 2B
\]
where $B$ is as above. 

Therefore, taking limits as $n$ grows to infinity, we obtain 
$|\gamma_m| \le m$. 

On the other hand, we have that $\gamma_m$ increases the $x$ coordinate 
by $m$. Since the metric \eqref{AFmetric} is bigger than 
$(1 + y^2) (dx^2)^{1/2}$ we obtain that the $|\gamma_m| \ge m $ and 
that this bound can only  be saturated when the segment is contained 
in the equator. See later Proposition~\ref{totalbudget} for a more
detailed argument. 

The lift of equator $\mathcal{E}$ in the universal cover of $\M$, is a cylinder 
$ \widetilde{\mathcal{E}} = \mathbb{S}^1 \times \mathbb{R}$.~\footnote{Note that $\widetilde{\mathcal{E}}$ is topologically non-trivial 
and homotopically non-trivial paths. Of course, these paths 
are homotopically trivial in $\widetilde{M}$. }
It is a flat cylinder (the universal cover of $\mathcal{E}$ is 
the flat plane).  
So that the orbits of minimal length are precisely the 
straight lines in the plane.  In particular, the minimizing geodesics
connecting two points are unique. 

Now, we can give several arguments to show 
that such $\gamma_m$ does not exist. 

We note that, because  $\omega \notin \mathbb{Q}$, the straight line 
connecting $(0,0,z_0)$ to $(0, 0, z_0 e^{2 \pi i m \omega})$ has length strictly 
bigger than $m$. 

Alternatively, observe that because the rotation $\omega$ is irrational, we have 
that straight lines connecting $(0, 0, z_0)$ and $(m,0, z_0 e^{2\pi i m \omega })$
are different for different $m$, so that the multiples of the minimizing orbit do 
not agree with each other. 

\begin{remark} 
Note that the above arguments use essentially that 
the equator, in spite of being a cylinder, carries no topology in 
the angle, since the turns around the angle can be undone by moving over 
the pole so that all the straight lines in the equator are homotopic 
in $\M$. 

If it was not because of the possibility of using  homotopy in the 
two dimensional sphere we would not be able
to verify the properties of the example. 

It would be tempting to try to do 
a similar construction in $\mathbb{T}^2$ by adding a twist in one
of the coordinates. This does not work because the topology 
$\mathbb{T}^2$ makes its lines with different slopes non-homotopic. 
Of course, such construction would contradict the results of 
\cite{Hedlund32}.
\end{remark}

\begin{remark}  
Given the characterization of the Tonelli Class-A minimizers, 
we can take $n$ so that the $e^{2 \pi i n m \omega}$ is arbitrarily close 
to $1$. When this factor is very close to $1$, the orbit connecting
a point $(0,0, z_0)$  to its 
translate $(0,0, z_0e^{2 \pi i n m \omega})$,  can be made 
shorter than the cover of the Tonelli orbit.

This makes concrete the fact that every Tonelli orbit has a cover 
which is not a minimizer.

Note that we if indeed $e^{2 \pi i n m \omega}$ is very close to $1$,
then, $e^{2 \pi i (n+1) m \omega}$ will be close to
$e^{2 \pi i m \omega}$, so that we can arrange also that the minimizers in
the $n+1$ cover approximate the original Tonelli orbit in segments 
of long length. 
\end{remark}

\begin{remark} 
One can think heuristically that the mechanism at play is the conflict between 
the natural minimizers which are the straight lines in the equator
and the fact that there is a twist given by \eqref{equivalence} 
which prevents the straight lines from being periodic. 
\end{remark}

\subsection{Relations between the topology and the existence of Class-A  geodesics}

The following result provides a further tie between Tonelli's theory, Morse's theory of Class-A geodesics and the topology of the manifold.
\begin{theorem}
  For any smooth metric in
  $\M = \mathbb{S}^1 \times \mathbb{S}^d$, $(d\geq 2)$, there is a Class-A geodesic.
\end{theorem}

Of course, the case $d=1$ is proved in \cite{Hedlund32} but the proof is very different from the proof presented here. Indeed, the proof presented
only works for $d \ge 2$. 

In general manifolds, the existence of Class-A minimizers is established in 
\cite{Mather91} (the orbits in the support of 
a minimizing measure are Class A). See also 
\cite{Bernard02}, but the proof requires sophisticated tools.

\begin{proof}
The universal cover of $\M$ is $\tilde{\M} = \mathbb{R}\times\mathbb{S}^d$ (here is where we use $d\geq2$).

Proceeding as in Tonelli's theorem, we can obtain a shortest geodesic
$\gamma_n$ by joining $\{-n\} \times \mathbb{S}^d$ to
$\{n\}\times \mathbb{S}^d$. (We can use Tonelli's theorem to obtain a
periodic curve which minimizes the length among periodic curves of
homotopy $2n$. Shift it back by $n$.)

All the $\gamma_n$ curves intersect $\{0\} \times \mathbb{S}^d$. We
observe that all the curves $\gamma_n$ are uniformly twice
differentiable because they satisfy the Euler-Lagrange equation for
geodesics. 

Hence, by Ascoli-Arzel\`a Theorem,
we can get a
subsequence $\gamma_{n_i} \rightarrow \gamma$
uniformly on compact
sets in the $C^1$ sense.

Following an argument in \cite{Morse24}, we show that $\gamma$ is a Class-A minimizer.
Assume that $\gamma$ was not a Class-A minimizer. Then, there is a $\tilde{\gamma}$ and $N\in \mathbb{N}$ which agrees with $\gamma$ outside of $[-N, N]\times \mathbb{S}^d$ such that
\[
\left| \gamma|_{[-N, N] \times \mathbb{S}^d}\right| \leq \left| \tilde{\gamma}|_{[-N, N] \times \mathbb{S}^d}\right| + \delta \qquad \text{ for some }\delta>0.
\]
Because of the convergence, we can find $n$ (without loss of generality, we assume $n\geq N$) such that 
\[
\bigg| \big| \gamma_n|_{[-N, N] \times \mathbb{S}^d} \big| - \big| \gamma|_{[-N, N] \times \mathbb{S}^d}\big| \bigg|  \leq \frac{\delta}{4}.
\] 
Hence
\[
\left| \gamma_n|_{[-N, N] \times \mathbb{S}^d}\right| \geq \left| \tilde{\gamma}_n|_{[-N, N] \times \mathbb{S}^d}\right| + \frac{3}{4}\delta.
\]
Therefore, we have that $\gamma_n$ would not be a minimizer for the length.
\end{proof}

Note that this proof works because of the topology of the manifold we use essentially that the universal cover has the structure of a line cross a compact manifold.

\section{Characterization of Class-A orbits in
the Almgren-Federer example.  Proof of Theorem~\ref{thm:classA}}
\label{sec:classA}

Theorem~\ref{thm:classA} will be proved by a sequence of propositions that keep on restricting the possibilities of Class-A geodesics.

The argument is very typical of similar arguments in calculus of
variations: We first show that there is a finite total budget for
deviations. Secondly, we prove that the fluctuations have to be
concentrated, each of them takes a substantial amount of the budget,
hence there are only finitely many of them. Finally, we show that the
fluctuations are none.

\subsection{A finite total budget for deviations}
\begin{proposition}\label{totalbudget}
Let $\eta_n$ be a segment of a Class-A geodesic joining a point of the form $(0, y_1, z_1)$ to another point $(n, y_2, z_2)$. Then, $|\eta_n| \leq n+ 2B$ where $B= \text{diam} \{(0,y,z) ~:~ y\in  [-1,1], y^2+ |z|^2 =1 \}$.
\end{proposition}

\begin{proof}
Just consider the curve formed by three segments as follows:
The first segment joins $(0, y_1, z_1)$ to $(0, 0, z_1)$, the second one joins $(0, 0, z_1)$ to $(n, 0, z_1)$ and the third segment joins $(n, 0, z_1)$ to $(n, y_2, z_2)$.

The first and third segments have length less than $B$ and the second one has length $n$.

Since $\eta_n$ is part of a Class-A geodesic, its length should be less than the trial segment.
\end{proof}

\begin{proposition}\label{notime}
Let $\eta_n (t) = (x(t), y(t) , z(t))$ be a segment of a Class-A geodesic as in Proposition~\ref{totalbudget}. Let $\delta>0$. Then, the set $O_{\delta, n} = \{ t~:~ |y(t)| \geq \delta \} \subset \mathbb{R}$ has measure
\begin{equation}\label{oscillationbound}
\left| O_{\delta, n}  \right|  \leq 2 \frac{B}{\delta^2}  ,
\end{equation}
where $|O_{\delta, n}|$ denotes Lebesgue measure of the set $O_{\delta, n}$.
\end{proposition}
\begin{proof}
It will be important that the bound~\eqref{oscillationbound} is uniform in $n$.
The proof is just to observe that 
\[
| \dot{\eta}_n (t) |_g \geq \left(1+ y^2(t) \right) \  |\dot{x}(t)|.
\]
By Proposition~\ref{totalbudget}, we have 
\[
\begin{split}
n+ 2 B &\geq \int_0^T \left| \dot{\eta}_n(t) \right| \ dt\\
&\geq  \int_0^T \left( 1+ y^2(t) \right) \ |\dot{x}(t)| \ dt\\
&\geq (1+ \delta^2) \int_{O_{\delta, n}}  |\dot{x}(t)| dt + \int_{[0, T] \setminus O_{\delta,n}} |\dot{x}(t)| \  dt\\
&= n + \delta^2 \  | O_{\delta,n}|. \qedhere
\end{split}
\]
\end{proof}

\subsection{Possible fluctuations for a finite number of intervals}
The next step shows that if there is a deviation from the equator it has to be a substantial one.
This is very analogue to density estimates in measure theory.

The key idea is that separating from the equator makes one to pay for deviating from the equator in the $\dot{x}$ component. The only way that this can pay off is if the penalty by measuring the weight of $\dot{x}$ is offset by derivatives in $\dot{z}$ component. This only happens if we get close to the north/south pole.

The following elementary side calculation 
makes more precise the \emph{hourglass shape} of the manifold $\M$. 
We recall that the length of a parallel at height $y$ is 
$2\pi(1+y^2) \sqrt{1- y^2}$. 
Write $u= y^2$ and let $f(u)= (1+u) \sqrt{1-u}$ and then,
\[
\begin{split}
f'(u) &= \sqrt{1-u} - \frac{1}{2} \frac{1+u}{\sqrt{1-u}}\\
&= \sqrt{1-u} \left( 1-u - \frac{1}{2} u -\frac{1}{2}  \right)
= \sqrt{1-u} \left( \frac{1}{2} - \frac{3}{2} u \right)
\end{split}
\]
So the function $f$ is increasing when $u\leq \frac{1}{3}$, that is, $|y| \leq \frac{\sqrt{3}}{3}$.

We find it convenient to write $z = \sqrt{1-y^2} e^{i \sigma}$ where $\sigma$ is the circle and it can be thought of as the longitude. Then,
\begin{equation}
\dot{z}(t) = \sqrt{1-y^2} e^{i \sigma(t)} i \dot{\sigma}(t) - \frac{y\ \dot{y}(t)}{\sqrt{1-y^2}} e^{i \sigma(t)} .
\end{equation}
So,
\[
|\dot{z}(t)| = \sqrt{ \frac{y^2}{1-y^2} \ \dot{y}^2(t) + (1-y^2) \ \dot{\sigma}(t)^2} 
\]
and the length is
\begin{equation}
\sqrt{(1+y^2)^2 |\dot{x}|^2  + \frac{(1+y^2)^2}{1-y^2} |\dot{y}|^2 + (1+y^2)^2 (1-y^2) \  \dot{\sigma}^2 }.
\end{equation}

\begin{proposition}\label{jumps}
Let $0<\delta \ll 1$. Consider any segment $\eta(t) =(x(t), y(t), z(t))$ that joins $(x_1, y=\delta, z_1)$ to $(x_2, y=\delta, z_2)$.
Assume $\delta \leq |y(t)| \leq \frac{\sqrt{3}}{3}$. Then, the segment $\eta(t)$ is not part of a Class-A geodesic.
\end{proposition}

The proof of Proposition~\ref{jumps} is very simple. 
We consider the rearranged path $\tilde{\eta}(t)$ defined by
\begin{equation}
\begin{split}
\tilde{x}(t) &= x(t);\\
\tilde{y}(t) &= \delta;\\
\tilde{z}(t) &= \frac{\delta}{y(t)} z(t).
\end{split}
\end{equation} Hence $\tilde{\sigma}(t) = \sigma(t)$.

It is clear that the rearranged path $\tilde{\eta}(t)$ and the original path $\eta(t)$ have the same original and final points and that the coefficients of the derivatives in the expression of the length have decreased.
\qed

As an immediate corollary of Proposition~\ref{jumps}, we obtain that if a Class-A geodesic leaves the region $|y| \leq \delta$, then it has to reach $|y| \geq \frac{\sqrt{3}}{3}$. Since the Class-A geodesics satisfy the geodesic equation, these excursions have to take a time that is bounded from below.
Note that this time is independent of $\delta$.

Therefore, so far, we have showed that all the possible Class-A geodesic stay within $|y|< \delta$ except for a finite number of intervals, in which they reach $|y| >\frac{\sqrt{3}}{3}$. This number is independent of the length considered. Therefore, for any $\delta>0$, the Class-A geodesics will have to include infinite intervals where $|y| < \delta$ and the intervals will be of the form $[a_+, \infty)$ and $(-\infty, a_-]$. 

Of course, we will take $\delta$ small enough such that the factor $(1+y^2) \sqrt{1-y^2}$ is monotone in $y^2$ for $y^2 < \delta$.

\subsection{Non-existence of fluctuations}
From now on, it looks that we can follow the previous argument. 
We are in better shape in some aspects, but we are missing something because we do not have the $(-\infty, \infty)$ intervals, but only semi-intervals $[a_+, \infty)$ and $(-\infty, a_-]$.

\begin{proposition}[crossing lemma]\label{upcrossing}
Take $0< \alpha< \delta$. If we have a segment of a Class-A geodesic with 
\[
y(t_1) = \alpha = y(t_2),\qquad y(t) \geq \alpha\quad \forall ~t\in[t_1, t_2].
\]
Then, $y(t) = \alpha$ for any $t\in [t_1, t_2]$.
\end{proposition}

\begin{proof}
We see that if we lower the $y$ to $\alpha$, we get a shorter line. 
\end{proof}

We therefore have that $\lim\limits_{y\rightarrow +\infty} y(t)=0, \lim\limits_{y\rightarrow -\infty} y(t)=0$. Both limits exist because of monotonicity and if they were different from zero, we would contradict Proposition~\ref{notime}.

In the following, we will rule out the bumps using the conserved quantities of
the geodesic flow.
See Appendix~\ref{sec:geodesic}. The conserved angular momentum (that comes
from the rotational symmetry) shows that if an orbit is rotating around the poles with a non-trivial rate (non-zero angular momentum) it cannot be
Class-A because rotating along the poles, at a certain rate will make
the orbit significantly longer that moving directly along the $x$ direction.

Finally, we study the orbits corresponding to angular momentum zero.

This  part of the argument uses slightly the geodesic flow and the dynamics
language.  This argument is very similar to the variational arguments that
show that, if there are any deviations from the 
conjectured easy solutions, they have to be large. We are
aware of some purely variational arguments along these lines
(these are usually done using density estimates), but they are
very cumbersome (we would love to hear about simple ones) and we hope
that some readers experts in variational 
methods may be motivated to learn about geodesic flows. 
The Mather theory of
minimizing measures has some dynamical components, that surprisingly
match very well with some aspects of the theory of flat chains
see \cite{Bangert95, Bangert99} for some comparisons 
between Mather's theory and geometric measure theory.

\smallskip

We use $L$ for the angular momentum (see Appendix~\ref{sec:geodesic})
and then we have:
\begin{proposition}
There are no Class-A orbits corresponding to $L \neq 0$.
\end{proposition}

\begin{proof}
We first note that $E\neq 0$ since all the points in the manifold are fixed points when the energy is zero.

We know  that all the orbits in the future (or in the past) stay in the region $|y| <\delta$. 
Note that this implies that $|\dot{\sigma}|$ is bounded from below in that region, that is 
\begin{equation}\label{lowerbound}
|\dot{\sigma}| \geq \frac{|L E|}{(1+\delta^2)^2} =: \sigma^*>0.
\end{equation}
For simplicity, we will only discuss the orbits in the future, the orbits in the past are obtained by changing $t$ into $-t$.
The lower bound of angular velocity \eqref{lowerbound} implies that after time $T\geq \frac{N 2 \pi}{\sigma^*}$, the orbit has gone around the equator $N$ times.
If $N$ is large enough, we can compare
the presumed Class-A orbit $\eta(t)$ with a segment obtained by joining the initial point to the equator (length $\leq \delta$) then joining $\dot{x} =1, \dot{y} =0, \dot{z}=0$ and then joining again.

This orbit is shorter when $\frac{N}{\sigma^*} \geq 2\delta$. Therefore, the orbit $\eta(t)$ is not Class-A.
\end{proof}

The final part of the argument is: 
\begin{proposition}
The only Class-A geodesics with $L=0$ are the straight orbits given in Theorem~\ref{thm:classA}.
\end{proposition}

The reason why this is true is intuitively clear. The only reason why it pays  off to depart from the equator is to gain by effecting the changes in $z$-direction around the poles.

If we have no angular momentum, there is no change in $z$-coordinate and therefore, leaving the equator is harmful. More formally, we note that if $\dot{\sigma}=0$, we are faced with a 2-dimensional system and the motion happens in a plane. Any excursion from the equator will deviate from the straight line.
\qed

\section{Some possible extensions of the argument and some applications} 
\label{sec:speculation}

\subsection{Building more complicated examples} 
Given a manifold $\M$, with universal cover $\widetilde{\M}$ and fundamental group $\pi_1(\M)$ 
(so that $\M= \widetilde{\M} / \pi_1(\M)$). 
Let $\ell: \pi_1(\M) \rightarrow \mathbb{R}$ be a cocycle of the fundamental group (i.e., $\ell(\gamma_1 \circ \gamma_2) = \ell(\gamma_1) + \ell(\gamma_2)$.) We can construct a manifold
\[
\mathbb{S}^2\times \tilde{\M}/\sim_{\ell}
\]
where $\mathbb{S}^2$ is parameterized as in \cite{Federer74}. We say that $(x, y, z) \sim (\tilde{x}, \tilde{y}, \tilde{z})$ if and only if we have that 
\begin{equation}
\begin{split}
\tilde{x} &= \gamma x\\
\tilde{y} &= y\\
\tilde{z} &= e^{i \ell(\gamma)} z
\end{split}
\end{equation}for some Deck transformation $\gamma$. Given a Riemannian metric $g$ on $\M$, we can define a metric on $\tilde{\M}$ by $(1+y^2) g$.

It is easy to check that if  $\ell(\gamma) \in \mathbb{R}\setminus \mathbb{Q}$ for some $\gamma\in \pi_1$, we have the same phenomenon as in the example of \cite{Federer74}: All the Tonelli geodesics in the classes $\gamma^n$ are not Class-A.

In particular, if 
\[
\ell(\gamma) \notin \mathbb{Q} \quad \forall~\gamma\in \pi_1(\M)
\]
then there are no Tonelli geodesics which are Class-A. Hence there are no periodic Class-A orbits.

An amusing example of this construction is: when we take $\M=\mathbb{T}^3$ endowed with the metric of Hedlund example.

Note that there is a relation between cocycle of the fundamental group and the cohomology of the manifold. See \cite{Brown82}. Hence, the language
in terms of cocycles is equivalent to the language in terms of cohomology.

\subsection{Tentative applications in Statistical Mechanics}
There is  an interesting physical interpretation of the Almgren-Federer
 example, which also makes it relevant 
for some problems in statistical mechanics.

We recall that the famous classical XYZ model of Heisenberg \cite{McCoy10} consists of a system of particles in the line. 
Each of this particle occupies a state 
described by  three coordinates $u_1, u_2, u_3$ constrained to 
$
u_1^2 + u_2^2 + u_3^2 =1.
$
That is  $u=(u_1, u_2, u_3)\in \mathbb{S}^2$.

The state of the whole system 
(called configuration) 
 is determined by describing the state  $u_i$ 
of the $i$ particle for all particles.
 So, a configuration is just
 a mapping  $\underline{u}: \mathbb{Z}\rightarrow \mathbb{S}^2$.

These particles interact with their neighbors and the substratum so that the system is described by an energy given by the functional:
\begin{equation}\label{energy}
E(\underline{u}) := \sum_{j\in \mathbb{Z}} S(u_j, u_{j+1})
\end{equation}
where $S$ is the interaction energy among 
next nearest neighbors. 
The sum in \eqref{energy}  is merely formal and is not meant to converge.

What physicists call ``ground state'' is identical to the Class-A minimizers of the calculus of variations.
They are configurations whose energy cannot be lowered by changing a finite number of sites. Note that the definition of Class-A/ground state   does not need that the 
sum in \eqref{energy} converges.

The example in \cite{Federer74} can be recast in the language of the XYZ model. We have that $\widetilde{ \mathbb{S}^2 \times \mathbb{S}^1 }$, the universal cover of $\mathbb{S}^2 \times \mathbb{S}^1$, is $\mathbb{S}^2 \times \mathbb{R}$.

A configuration $u_j$ can be thought of as sequence of points in $\mathbb{S}^2 \times \mathbb{R}$, given by $(u_j, j)$. If we define
\begin{equation}
  \label{interactionXYZ}
S(u_j, u_{j+1}) := \text{dist}\left((u_j, j), (\mathcal{R} u_{j+1}, j+1) \right)
\end{equation}
where $\mathcal{R}$ is the rotation by $\alpha$ in the identification.
We see that the minimizers of the functional $E$ will be geodesics.

The results of \cite{Federer74} can be expressed in statistical
mechanics jargon as saying that imposing periodic boundary condition,
we never obtain a ground state. Sliding the boundary condition
further, will always lower the energy and the system with periodic
boundary condition never becomes a ground state.  
A configuration corresponding to a Tonelli orbit will become 
destabilized under fluctuations of longer and longer periods. 
It would be interesting to understand better the dynamics of
these relaxations. 

The
form of the interaction \eqref{interactionXYZ} is natural in
quasi-periodic media so that advancing one index $i$ by $1$ is
equivalent to rotating an internal phase.

The statistical mechanics interpretation, makes it natural 
to consider several extra features such as many-body interactions,
long range interactions whose consequences are interesting to explore.

\appendix

\section{The dynamical point of view:geodesic flows}
\label{sec:geodesic}

The problem of existence of geodesics with certain properties can
be recast in a more dynamical language leading to \emph{geodesic flows} 
\cite{Caratheodory89,Paternain99,AbrahamM78,ArnoldA68}. In this paper, we rely much more
on the methods of the calculus of variations, but we want to remark
that the Almgren-Federer example also has interesting properties
from the dynamical point of view.

\subsection{Basic definitions} 

Given a function $L: T\M  \rightarrow \mathbb{R}$ one can consider
the functional
\[
  S_{t_1}^{t_2}[\gamma]  = \int_{t_1}^{t_2} L( \gamma(t), \dot \gamma(t)) \, dt \, .
\]
Under appropriate regularity conditions -- we refer to the references above --
the functional $S$ is differentiable among paths that satisfy
$\gamma(t_1) = a; \gamma(t_2) = b$. A path $\gamma$ is a critical
point of the functional if and only if it satisfies
the Euler-Lagrange equation.
\begin{equation}
  \label{EulerLagrange}
  D_1 L - \frac{d}{dt} D_2 L = 0.
\end{equation}

The  first order differential equation  \eqref{EulerLagrange}
in $T\M$ can be interpreted as a second order
equation in $\M$, which is considered as an evolution equation. Again,
under appropriate conditions on the regularity, growth, of $L$ and on the
manifold $\M$, the flow is complete (given any initial condition in $T\M$,
there is a solution of \eqref{EulerLagrange} defined for all times.)

Much of the theory works also for time-dependent Lagrangians, but we will
not consider this.

\subsubsection{The geodesic flows} 
When we take
$L_1 (\gamma, \dot \gamma)  = g_{\gamma}(\dot \gamma, \dot \gamma)^{1/2}$
the functional $S$ is just the length.

However, much of the theory of geodesic flows is obtained
taking the Lagrangian
$L_2(\gamma, \dot \gamma)  = g_{\gamma}(\dot \gamma, \dot \gamma)$.

The   Lagrangian $L_2$, quadratic in the velocity
is more natural in Mechanics\footnote{In Mechanics, it is customary 
to multiply it by a factor $\frac{1}{2}$.} and,
as we will see later, the superlinear growth of the Lagrangian
in the velocity is important for the Mather theory of
homology minimizing measures (see Appendix~\ref{sec:mather} and the references
there).

As it turns out, the solutions of the Euler-Lagrange equations
for both systems are the same. 

The reason is that,
using that $L_2$ is homogeneous in the velocity, one can show
that $L_2$ is a conserved quantity for the Euler-Lagrange flow
corresponding to $L_2$.  For physicists, the $L_2$ is the kinetic
energy. The fact that the energy is the same as the Lagrangian
depends on the fact that the energy is homogeneous in the velocity. 

If we consider a Lagrangian $\tilde L = F(L_2)$ with $F$ a smooth function
(we omit some details on regularity, etc.)
we obtain that the Euler-Lagrange equations corresponding to it
are
\begin{equation} \label{modified} 
  F'(L_2 ) D_1 L_2 - \frac{d}{dt}(  F'(L_2) D_2 L_2)  = 0.
\end{equation}
We can see that if $\gamma(t)$ is a solution of the Euler-Lagrange equation
for $L_2$, then, because $L_2$ is conserved for these solutions, then
$\gamma(t)$ solves \eqref{modified}. Since these solutions can match
all the initial conditions, they are all the solutions.

\subsection{Conserved quantities of the geodesic flow 
  of the metric \eqref{AFmetric}}

The metric \eqref{AFmetric} has 3  local symmetries
(invariance with respect to time, horizontal translations
and rotations along the longitude. According to Noether's 
theorem, this means that the system has 3  conserved quantities.

It is interesting to note that, due to the gluing in \eqref{equivalence} 
even if we have infinitesimal symmetries, the symmetry of translation 
along the $x$ axis does not have compact leaves. This is closely related 
to the fact that there are no strict minimizers.

As remarked above, we could apply Noether's theorem
either to the Lagrangian $L_1$ or to the $L_2$, the conserved
quantities obtained are different in both cases, but they are
function of each other.

\begin{proposition}
  The geodesic flow corresponding to
  the Lagrangian $L_1$ corresponding to 
the metric \eqref{AFmetric} preserves the following conserved quantities:
\begin{description}
\item [energy]
\[
\begin{split}
E&=(1+y^2) \sqrt{\dot{x}^2 + \dot{y}^2 + |\dot{z}|^2}  \\
  &=(1+y^2) \sqrt{\dot{x}^2 + \dot{y}^2 + (1-y^2) \ \dot{\sigma}^2 }
\end{split} \]

\item [momentum]
\[
\begin{split}
P&= (1+y^2) \frac{\dot{x}}{\sqrt{\dot{x}^2 + \dot{y}^2 + (1-y^2)\  \dot{\sigma}^2}} \\
&= \frac{(1+y^2)^2 \  \dot{x}}{E}
\end{split}
\]

\item [angular momentum]         
\[
\begin{split}
L = \frac{(1+y^2)(\dot{z}_1 z_2 - \dot{z}_2 z_1)}{2 \sqrt{\dot{x}^2 + \dot{y}^2 + |\dot{z}|^2} } \\
&= \frac{(1+y^2)^2 (1-y^2)\  \dot{\sigma}}{E}
\end{split}
\]
\end{description}
\end{proposition}

\begin{proof}
(i). The energy is conserved because of the well-known fact that the length is preserved by the geodesic flow denoted by $\phi^t$.
In fact, note that applying Noether's theorem (see, for example \cite{Arnold89}) for the one-parameter group of diffeomorphisms $\phi^t$ (the geodesic flow)  produces a time conserved quantity:
\[
\frac{\partial E}{\partial \dot{q}} \dot{q} = E
\]
where $q=(x,y,\sigma)$.

(ii). The momentum corresponds to the translation:
\begin{equation}
\begin{split}
x &\rightarrow   x +\epsilon\\
y &\rightarrow y\\
z &\rightarrow z.
\end{split}
\end{equation}
Applying the general rule of Noether's theorem, we obtain the first integral $\frac{\partial E}{\partial \dot{x}} =P$.

(iii). The angular momentum corresponds to the transformation:
\begin{equation}
\begin{split}
x &\rightarrow   x \\
y &\rightarrow y\\
z &\rightarrow z e^{i \alpha}  \Leftrightarrow \sigma\rightarrow \sigma + \alpha.
\end{split}
\end{equation}
Applying Noether's theorem, we obtain the first integral $\frac{\partial E}{\partial \dot{\sigma}}= L.$
\end{proof}

Note that these three first integrals will allow us to integrate the equations in the universal cover. For example, we can obtain $\dot{y}$ as a function of $y$ and the conserved quantities:
\[
\begin{split}
E^2 &= (1+y^2)^2 [\dot{x}^2 + \dot{y}^2 + (1-y^2) \dot{\sigma}^2]\\
&= (1+y^2)^2 \left[  \frac{P^2 E^2}{(1+y^2)^4}  + \dot{y}^2 + \frac{L^2 E^2}{(1+y^2)^4 (1-y^2)}\right].
\end{split}
\]

Even if one can get the explicit solutions, it is not easy to understand
the minimizing properties of the orbits. Of course, we will present arguments 
from the calculus of variations later.

\begin{remark}
It is interesting to remark that the symmetry under translation is a differentiable symmetry which leads to conserved quantity by Noether's theorem, but it does not allow to make the 
quantities descend to 
the manifold. (The orbits of the symmetry are dense.) Hence, we cannot find quotient manifolds corresponding to the system.

This situation happens very often in quasi-periodic systems (see for example \cite{PetrovL99, SuL12}).
\end{remark}

\subsection{On the notion of integrability and the integrability of
  the geodesic flow  for the Almgren-Federer example}
\label{sec:integrability}

The word \emph{integrability} is used often very loosely
in Hamiltonian mechanics and, a system can be integrable or
not depending on the precise meaning given to integrability.

One of the loosest definition is that the system is integrable if
it has as many conserved \emph{``functionally independent''} conserved
quantities as degrees of freedom. In this sense, the geodesic
flow of the Almgren-Federer example is integrable.

On the other hand a more strict definition of integrability is 
that the conserved quantities commute with each other and
be independent everywhere.  It is known by the Liouville-Arnold
theorem \cite[p. 278]{HoferZ11} that this implies  that the phase space should
be of the form $\mathbb{T}^d \times \mathbb{R}^d$.  Clearly,
the phase space of the Almgren-Federer example (in the Hamiltonian 
formalism) is 
 $T^*( \mathbb{S}^1 \times \mathbb{S}^2)$
is not of this form, so it cannot be integrable in this stronger sense.

It seems that the Hamiltonian flow of 
the Almgren-Federer  example 
 satisfies this definition of integrability, in some open sets
but that there are singular leaves.  See \cite{BolsinovF04} for a discussion
the topological issues involved in integrable systems with singular 
leaves.

The variables produced in the Liouville-Arnold theorem are called
action-angle variables. They can be obtained by integrals. We note however
that, even for polynomial systems, the action variables have complex
singularities (even for an anharmonic oscillator or a pendulum), so that in many
computations it is advantageous to have methods that do not rely on the action
angle variables.

In some notions of integrability, it is also required that the
integrals or the action-angle are  algebraic functions of
the coordinates or obtained through some specific method, etc.

It is unfortunate that in many discussions of integrable systems,
precise definitions are omitted.

\section{H. M. Morse's theory of globally minimizing geodesics}
\label{sec:morse} 

In this section, we present a summary of the main concepts and results 
of \cite{Morse24}. The  paper \cite{Morse24}  developed the global
theory of minimizing geodesics in dimension 2 except for the torus and the sphere. 
The global 
theory on the torus was developed in \cite{Hedlund32} and in 
the sphere it is meaningless. 

 The effect of the example
\cite[Section 5.11]{Federer74} (and of Hedlund's example
\cite{Hedlund32}) is to show that this theory of \cite{Morse24} on
geodesics does not generalize to higher dimensions.

Let $\gamma: [a,b]\rightarrow \M$ be an absolutely continuous curve in a Riemannian manifold $(\M, g)$; then the length is defined as follows:
\[
|\gamma|_g : = \int_a^b | \dot{\gamma}(t) |_g dt = \int_a^b \sqrt{g(\dot{\gamma}(t), \dot{\gamma}(t))} dt.
\]

\begin{definition}[Class-A geodesics]
\label{def:classA}

A geodesic $\gamma$ parameterized by the arclength $s\in\mathbb{R}$ is
said to be Class-A with respect to the Riemannian metric $g$ if for
any point $s_1< s_2$, the absolutely continuous geodesic segment
$\tilde{\gamma}$ with end points $\gamma(s_1)$ and $\gamma(s_2)$
homotopic to $\gamma([s_1, s_2])$ on $\M$, we have
\[
\left|\gamma|_{[s_1, s_2]}\right|_g \leq \left|\tilde{\gamma}\right|_g.
\]

\end{definition}

One of the main results of \cite{Morse24}
is the 
the following
\begin{theorem} Let $\M$ be a 2-dimensional Riemannian manifold of genus 
$g \ge 2$. Then, by the  uniformization theorem, the universal cover $\tilde\M$ 
is a disk and the metric in $\tilde\M$ is conformal to the standard Poincar\'e 
metric. 

There exists a constant $C>0$ such that for every geodesic $\alpha$  of the standard Poincar\'e metric in the Poincar\'e disk, there is Class-A geodesic $\gamma_\alpha$ in $\widetilde\M$ such that $\dist(\alpha, \gamma_\alpha) \le C$.

If $\M = \mathbb{T}^2$, for any straight line $\ell$ in the universal cover
(which we identify with $\mathbb{R}^2$)
there exists a Class-A geodesic $\gamma_\ell$ such that 
$\dist(\ell, \gamma_\ell) \le C$. 
\end{theorem} 

The proof of the above results in \cite{Morse24} is very readable. 
It is based on the fundamental lemma that two minimizing geodesics 
can only cross once. (This lemma generalizes to many other contexts
and is an important inspiration for the analysis in 
\cite{AubryD} of the ground states in Frenkel-Kontorova model.
See \cite{Bangert88} for a very general point of view that 
uncovers the deep relation among these seemingly disparate areas. 

\begin{remark}
The examples of \cite{Hedlund32, Federer74} show that the results of
the theory of geodesics do not generalize to geodesics in higher dimensions.

On the other hand, as pointed in \cite{Moser86,Moser03} there are
rather satisfactory generalizations to \emph{Objects of codimension 1}.
Of course, the analysis becomes much more complicated since one has
to deal with objects of higher dimension which require
tools more sophisticated than ODE's. This has lead to many results
by different schools.

Many results in elliptic PDE's on tori, appear in
\cite{RabinowitzS11}.  The paper \cite{CaL01} considers
minimal surfaces of codimemsion 1 (and even elliptic integrands) 
and allows rather general manifolds
(the fundamental group has to be residually finite). The notion of
minimizers there is related to cocycles, which by \cite{Brown82} is
also related to cohomology.  The survey \cite{Bangert99} considers 
also relations to geometric measure theory. The papers 
\cite{Blank, KochLR, CandelL, LlaveV,CaffarelliL05} consider different 
generalizations to statistical mechanics. Many other codimension 
one contexts have been obtained in the literature, including 
nonlocal operators. 
\end{remark}

\section{Tonelli theory of minimizing periodic orbits} 

\label{sec:tonelli}
The  work of Tonelli introduced many of the now standard 
semicontinuity arguments in the calculus of variations. 

\begin{theorem}[Tonelli]\label{Tonelli theorem}
Let $\M$ be a compact Riemannian manifold without boundary. 

For every homotopy class $c$, there is  a closed curve $\gamma_c$ of 
homotopy class $c$ such that 
\begin{equation}\label{eq:minimizer} 
|\gamma_c| = \min_{\gamma \text{ has homotopy class } c}\{ | \gamma | \}.
\end{equation} 
\end{theorem}

\begin{definition}\label{def:tonelli}
We will refer to the curves  that satisfy \eqref{eq:minimizer} as 
Tonelli curves or Tonelli minimizers. 
\end{definition}

By today's standard, the argument is rather standard. 
There are many variants of the proof. But all have more or less
the same ingredients:
\begin{enumerate}
\item
  The set of curves in a homotopy class
  with a finite total length  less than $A$ is pre-compact in $C^0$ topology.
  (Note that bounding the total length gives a bound in the oscillation).

  Of course, the set is not empty, for $A$ sufficiently large.

There are many precise formulations and can use some versions of
  Sobolev embedding theorem, Ascoli-Arzela, broken geodesics, etc. 
\item
  The length of curves is lower semicontinuous with respect to $C^0$
  convergence.

  This is a very simple argument if one defines the length as the supremum
  of the sum of the lengths of all the piecewise 
  geodesics joining points in the curve which are at sufficiently 
small distance (say, one half of the distance given by the Hopf-Rinow 
theorem). 

\item
  The passing to the limit in a minimizing sequence does not change the
  homotopy class. 
\end{enumerate} 
Many more details can be 
found in \cite{Mane90, Mather91, Mazzucchelli, Fathi97}.

\section{J. Mather's theory 
  of homological minimizing measures}
\label{sec:mather} 

An important point of view on minimization problem for 
geodesics (and more generally Lagrangians) was introduced 
by J. Mather in the 80's and 90's \cite{Mather89, Mather91, MatherF}. Important generalizations
were obtained  by R. Ma\~n\'e \cite{Mane90, Mane92, Mane96,  Mane97, ContrerasDI97, ContrerasI99}.
See also  \cite{Mazzucchelli,Sorrentino15}.

Of course we cannot even summarize the main concepts of 
this rich theory and we refer to the references above.
(The theory applies also to periodic Lagrangians, 
we omit even mention many important concepts 
such as the  critical value, dual formulation, 
 ergodic characterizations, etc.)

In this section, we will only discuss the 
minimizing measures point of view and how it relates to 
the Almgren-Federer example.   There are some resemblances
between the motivation of \cite{Federer74} and \cite{Mather89}: the 
need to consider minimizers in dual spaces, the role of cohomology.

\subsection{Some definitions in the   theory of  $c$ minimizing measures}

In this section, we go over briefly the main definitions. 
We will assume that all functions 
are differentiable enough  in this informal presentation. We 
refer to the references above  for  precise treatments.
Our main goal in this paper is just to present the main concepts as
applied to the Almgren-Federer example and we are omitting many
important precisions (regularity assumptions, precise convexity assumptions,
completeness of the flow, etc.)

In this section it will be crucial that we use the Lagrangian
$L_2$ from Appendix~\ref{sec:geodesic} since the superlinear
growth will be important.

We  recall that a one form $\eta$ is a function
$\eta: T\M \rightarrow \real$, which is linear on the tangent directions. 

Hence we can add a form to a Lagrangian. 
\[
L_\eta = L  - \eta.
\]
(Of course, the minus sign above is a convention, which simplifies 
some of the calculations). 

It is well known that  if $\eta$ is a closed form, the critical points 
of $L_\eta$ are the same as those of $L$. Indeed, if 
$\eta$ is a closed one form, the difference 
between the action of $L$ and $L_\eta$ are  
terms that depend only on the ends of $\eta$, which are fixed
in the variational process.\footnote{
In the Physics literature this is sometimes described as ``adding 
a complete differential''.} Hence, the Euler-Lagrange equations for 
$L_\eta$ are the same as those for $L$.

In other words, the curves that are critical points for 
the action corresponding to $L_\eta$ are the same as 
the curves that are critical points for the action corresponding to 
$L$. 
On the other hand, the curves minimizing the action of
$L_\eta$ can be different from 
those of $L$.

The main concept of \cite{Mather89,Mather91,Mane90} is to  consider 
variational principles not on orbits but on probability measures

\begin{equation}
  \label{measureaction}
A_{\eta}[\mu]  =  \int_{T\M}  L_\eta \, d \mu.
\end{equation}

Note that to be able to define \eqref{measureaction} comfortably,
it is important
that the $L$ grows faster than linearly in the velocities. To apply
lower semicontinuity arguments, it will also be important that the
Lagrangian is convex in the fibers of $T\M$. 

\def\MM{ \mathfrak{M}}

One conceptually very crucial observation of \cite{Mather89,Mather91} is 
that $A_\eta$, depends only on $c =  [\eta]$ 
cohomology class of $\eta$. 

Hence, we will henceforth use the notation 
$A_c$, where $c$ is a cohomology class.

In \cite{Mather91} the infimum was taken over the sets of
probability measures that were invariant under the geodesic flow.
The paper \cite{Mane92} considers minimizers over 
all possible probability measures and shows that one obtains the same 
value because all minimizing measures are
invariant. 

Note that in both cases, by lower semicontinuity, etc arguments, 
one finds that there are minimizers in the class of measures considered.

\begin{remark} 
We note that there is a dual formulation. One can prescribe 
a rotation number for a measure (which is an element in homology)
and one can consider the minimization over measures with this homology. 
See \cite{Carneiro, Sorrentino15, Fathi97}.  From our 
point of view -- computing the minimizers in a concrete example -- 
it seemed easier to use the unconstrained minimizers, since 
minimizing with the constraint is harder to write up. 
\end{remark}

A very important function in Mather's theory is
\[
  \alpha(c) = \inf_\mu A_c [\mu].
\]
Note that for a fixed $\mu$, $A_c[\mu]$ is linear in $c$, so
that $\alpha[c]$, being infimum of linear functions 
 is a concave function.

\subsection{Characterization of Mather's $c$ minimizers for 
the Almgren-Federer example.} 

In the Almgren-Federer example, the space of cohomology is just the reals.
A good representative of the cohomology class is $c dx$
and any element of the class can be written as $c dx + dS$ where
$S$ is a function on the manifold. 

\begin{lemma} 
In the Lagrangian $L_2$ corresponding to the 
metric \eqref{AFmetric}, 

The $c$ minimizing measures for $c \ne 0$ are 
the measures concentrated on the set $y = 0$, 
$L_2  = \frac{1}{4}c^3$ and the velocity is along the $x$ direction only. 
$v_x =  \sqrt{L_2}$

The $c$  minimizing measures for $c=0$
are those with velocity zero and any distribution
in the position variables.
\end{lemma}

\begin{proof}

In this case, the $L_2$ is a conserved quantity of 
the geodesic flow. For a number $E$, the set of points
in $T\M$ that satisfy $L_2 = E$ is a manifold.~\footnote{We abuse the letter $E$, since we had also 
used it in Appendix~\ref{sec:geodesic}. Also, we note for physicists 
that there is a factor $2$ compared with customary definitions of 
the kinetic energy.}
We disintegrate the measure along the 
level surfaces of $L_2$ -- which is natural since it is a conserved 
quantity -- and furthermore disintegrate along $y$ levels. 
We write the problem 
as 
\begin{equation} 
\label{disintegrated} 
\int \,  d\alpha(E) \int \,d\beta_E(y_0) \int_{L_2 = E; y = y_0}  
\, d \nu_{E,y_0} 
( L_2  - c v_x  ) 
\end{equation} 
and we choose the conditional measures $\alpha, \beta, \nu$ 
so as to minimize
the integral. Note that there is some ambiguity in the 
integration. We can multiply by a factor $\nu$ and by the 
inverse factor $\beta$, etc. So that we can assume 
that $\nu_{E, y_0}$ is either zero or a probability measure.   
We can assume that it is always a probability measure and that the 
measures $\beta, \alpha$ give no weight to the $E, y_0$ for 
which $\nu$ is zero.

Since $d \nu_{E, y_0}$ is in the level set $L_2 = E$, for $c > 0$,
the integrand reduces to $E - c v_x$. If we had any measure $\nu$ we 
could do better by another measure which is concentrated as much
as possible in the  value of $v_x$. that makes the integrand
smaller. For $c > 0 $ -- which is the case that we will consider 
for the moment-- we want to take $v_x $ as large as possible 
consistent with being in the energy surface and a parallel. 
Since $E = (1+y_0^2)^2( v_x^2 + v_y^2 + v_z^2)$, we 
see that the optimal value to concentrate $v_x$ is 
$\sqrt{E}/(1 + y_0^2)$. 

Hence, we can assume that  for the minimizing measure $\nu$, we 
have:
\[
\int_{L_2 = E;n y = y_0}  
\, d \nu_{E,y_0} 
L_2  - c v_x =   E - c \sqrt{E}/(1+y_0^2) .
\]

Now we consider the measure $d \beta_E(y_0)$. Since $E, c$ are
fixed positive numbers, we can again assume that $d \beta$ is 
concentrated on the $y_0$ for which the integrand is smallest. 
That is, $y_0 = 0$. 

Hence, we are reduced to studying 
\[
\int \, d\alpha(E)  E - c \sqrt{E} .
\]

Again, clearly, it is good to concentrate the measure in the values 
of $E$ that make the integrand the smallest. An elementary calculation 
gives that this minimum is reached for $E = \frac{1}{4}c^2 $. 

The calculation for $c < 0$ is very similar and we omit the details. 

As for the case $c= 0$, we see that the minimization of the integrand 
happens precisely when all the velocities are zero. In this case 
the distribution in the space variables does not matter. 

\end{proof}

\subsection{Weak KAM theorem} 
There is a close connection between Mather's theory and weak KAM solutions of Hamilton-Jacobi equation:
\begin{equation}
H_2(x,du(x) + c )=\bar H(c), \quad x\in \M  \label{equation:PDEcellEquation} 
\end{equation}
where $H_2$ is the Hamiltonian associated to $L_2$. This equation is a
degenerate PDE equation of first order with two unknowns 
$(\bar{H}(c),u)$. The constant $\bar{H}(c)$ is unique for any given $c$ and is
called the effective Hamiltonian (also called Mather's
$\alpha$-function). The function $u$ defined on $\M$ is $C^0$
but may not be unique.One can refer \cite{Fathibook} for the weak KAM
theory.  Roughly speaking, the minimal geodesics can be embedded into
the characteristic fields of Hamilton-Jacobi equation.

\begin{definition}[Calibrated curve]
We say that  a  curve 
$\gamma: \mathbb{R}  \rightarrow \M$ is calibrated if for any $t, t' \in I$ with $t\leq t'$, we can find a real function $u$ such 
that 
\[
u(\gamma(t')) - u(\gamma(t)) = \int_t^{t'} L_\eta(\gamma(s), \dot{\gamma}(s)) ds + \bar{H}(c) (t'-t).
\]
\end{definition}

\begin{proposition} 
The Class-A geodesics in the Almgren-Federer example are calibrated
\end{proposition}

We note that the closed form $\eta$ is given by $\eta = c dx + d S$ 
where $S$ is a function on the manifold $\M$.

Since the energy is conserved, 
\[
\int_t^{t'} L_\eta(\gamma(s), \dot{\gamma}(s)) ds  = 
E (t- t') - c v_x (t - t') - S(\gamma(t')) + S(\gamma(t)) .
\]

One of the consequences of the general theory \cite{Fathibook} 
 is that the calibrated 
curves are Class-A.

\bibliographystyle{alpha}
\bibliography{reference}
\end{document}